\documentclass[a4paper,11pt,pdf]{amsart}
\usepackage{enumerate, amsmath, amsfonts, amssymb, amsthm, thmtools, wasysym, graphics, graphicx, xcolor, frcursive,comment,bbm}
\usepackage[colorlinks=true,citecolor=cyan,backref=page]{hyperref}
\usepackage{tikz}
\usepackage{rotating}
\usepackage{etex}
\usepackage{lscape}
\usepackage{tikz-cd}

\makeatletter
\newcommand{\thickhline}{%
	\noalign {\ifnum 0=`}\fi \hrule height 1pt
	\futurelet \reserved@a \@xhline
}

\definecolor{darkblue}{rgb}{0.0,0,0.7} 
 
\definecolor{darkred}{rgb}{0.7,0,0} 
\usepackage{hyperref}
\usepackage[all]{xy}
\usepackage[T1]{fontenc}
\usepackage{adjustbox}
\usepackage{enumitem}
\usepackage{vmargin}            
\setmarginsrb{3cm}{2.5cm}{3cm}{2.5cm}{0cm}{0.6cm}{0cm}{0cm}

\usepackage{caption,lipsum}
\usepackage{graphicx}                  
\usepackage{pstricks,pst-plot,pst-text,pst-tree,pst-eps,pst-fill,pst-node,pst-math}
\usepackage{setspace}
\usepackage{multicol}

\newcommand{\darkred}{\color{darkred}} 
\newcommand{\defn}[1]{\emph{\darkred #1}}

\usepackage{etex}

\newtheorem{theorem}{Theorem}[section]
\newtheorem{prop}[theorem]{Proposition}
\newtheorem{lemma}[theorem]{Lemma}
\newtheorem{cor}[theorem]{Corollary}

\theoremstyle{definition}
\newtheorem{definition}[theorem]{Definition}
\newtheorem{rmq}[theorem]{Remark}
\newtheorem{exple}[theorem]{Example}

\numberwithin{equation}{section}

\title[Interval Garside groups arising from involutions in reflection groups]{Interval Garside groups arising from involutions in finite reflection groups}

\author{Eirini Chavli}
\address{Université de Tours, IDP, CNRS UMR 7013, Faculté des Sciences et Techniques, Université de Tours, Parc de Grandmont, 
	37200 Tours, France}
\author{Thomas Gobet}
\address{Université Clermont Auvergne, LMBP, CNRS UMR 6620, Campus des Cézeaux, 3 place Vasarely, TSA 60026, CS 60026, 63178 Aubière cedex, France}

\begin{document}
	\maketitle
	
	\begin{abstract}
		We identify and study the interval Garside groups arising from the restriction of the absolute order on a Coxeter group to a lattice $[1,w]_T$, where $w$ is an involution. Those involutions $w$ for which $[1,w]_T$ is a lattice were previously classified by the second author; every such involution lies in the center of the parabolic subgroup generated by $[1,w]_T$. Except in type $B_n$, the obtained groups are isomorphic to (decomposable) right-angled Artin groups. We also investigate the situation for some finite complex reflection groups, mostly in rank two, taking for $w$ a (not necessarily involutive) central element. 
	\end{abstract}
	
	\tableofcontents
	
	\section{Introduction}
	
	Garside groups were introduced by Dehornoy and Paris~\cite{DP} as groups of fractions of monoids with good divisibility properties, generalizing the properties exhibited by Garside for the $n$-strand braid group~\cite{Garside_69}. They provide a framework to solve most combinatorial questions on infinite groups resembling braid groups (such as the word and conjugacy problems, the determination of the center, the absence of torsion, etc.).  
	
	The aim of this paper is to study certain Garside groups arising from finite reflection groups, using the machinery of so-called \textit{interval groups} (see Section~\ref{sec:pre} below for precise definitions), in the spirit of the dual Garside structures introduced by Bessis for Artin groups of spherical type~\cite{Bessis}. This framework allows one to realize the classical and dual Artin groups of spherical type as the interval Garside groups attached respectively to the interval between $1$ and $w_0$ in the weak order on the attached Coxeter group $W$, and the interval between $1$ and $c$ in the absolute order on $W$; here $w_0$ stands for the longest element in $W$, while $c$ stands for a Coxeter element. The fact that these intervals are lattices, together with a special property of the top elements, guarantees that the corresponding interval groups are Garside groups. 
	
	On one hand, showing that a given group $G$ is a Garside group is a hard task in general. On the other hand, the theory of interval groups gives a construction allowing one to produce families of Garside groups from known groups. More precisely, given a (finite) group $H$ and a system $S$ of elements generating $H$ as a monoid, one can endow $H$ with a natural partial order $\leq_S$ depending on $S$, and consider intervals $[1, w]_S$ in this partial order; here $w\in H$ is a particular element of $H$ (called \textit{balanced}). If this interval is a lattice, then it produces a Garside group $G([1,w]_S)$, of which $H$ is a quotient.  
	
	When considering the absolute order on a finite Coxeter group $W$, obtained using the set of generators of $W$ given by the whole set $T$ of reflections, every element $w$ is automatically balanced. It is thus natural to identify those elements $w\in W$ for which $[1, w]_T$ is a lattice, and to study the attached Garside groups. As recalled above, if $w$ is a Coxeter element, then $G([1,w]_T)$ is the dual Artin group attached to $W$, which is isomorphic to the usual Artin group. One advantage when considering $T$ as set of generators is that every element $u\in W$ is automatically balanced. It thus suffices to show that $[1, u]_T$ is a lattice to produce a Garside group.   
	
	In a previous work, the second author classified those involutions $u$ in Coxeter groups for which $[1, u]_T$ is a lattice: 
	
	\begin{theorem}[{\cite[Corollary~4.3]{Gobet_invol}}]\label{thm_gob_class}
		Let $W$ be a Coxeter group and $u\in W$ and involution. Let $P(u)$ be its parabolic closure, which is finite. Then $[1, u]_T$ is a lattice if and only if every irreducible component of $P(u)$ is of one of the following types: $$A_1, I_2(2k)\text{~with~}k\geq 2, B_n~\text{with~}n\geq 3, D_4, \text{~or~}H_3.$$
	\end{theorem}
	
	The aim of this paper is to identify the attached Garside groups (note that for $A_1$ it is trivially $\mathbb{Z}$, and the Garside structure coincides here with both the classical and dual ones on the braid group on two strands). To this end, it suffices to deal with the case where $P(u)$ is irreducible, as in the general case the group will be the direct product of those attached to its irreducible components. We recall some observations from~\cite{Gobet_invol} which facilitate this study. The poset $[1, u]_T$ fully lies inside $P(u)$, as $T$-reduced expressions of an element in a parabolic subgroup have all their factors in this parabolic. We can thus assume that $W=P(u)$, which forces $u=w_0$ (see~\cite[Corollary 3.5]{Gobet_invol}). 
	
	After collecting some prerequisites in Section~\ref{sec:pre}, the identification of the attached Garside groups is achieved in Section~\ref{sec:id} below; in every type except in type $B_n$, the obtained Garside groups are given by decomposable right-angled Artin groups. More precisely our main result is the following:
	
	\begin{theorem}\label{thm:main}
		Let $w_0$ be the longest element in one of the finite Coxeter groups $W$ listed in Theorem~\ref{thm_gob_class}. The attached interval Garside group $G([1,w_0]_T)$ is isomorphic to
		\begin{enumerate}
			\item $\mathbb{Z}\times F_k$ if $W$ is of type $I_2(2k)$, $k \geq 2$, where $F_k$ is the free group on $k$ generators, 
			\item $\mathbb{Z} \times G_n$ if $W$ is of type $B_n$, where $G_n$ is the group defined by the presentation in Proposition~\ref{prop_bn_clean} below (see also Example~\ref{ex_bn} for the case $n=3$), 
			\item $\mathbb{Z} \times \mathrm{RAAG}(\triangle ~ \triangle ~ \triangle)$ if $W$ is of type $D_4$, 
			\item $\mathbb{Z} \times \mathrm{RAAG}( ~| ~| ~ | ~ | ~ |~)$ if $W$ is of type $H_3$.
		\end{enumerate}
		Here, in the last two cases, $\mathrm{RAAG}(\mathsf{Gr})$ denotes the right-angled Artin group with graph $\mathsf{Gr}$, with the convention on the graph that an edge denotes a pair of commuting generators. 
	\end{theorem}
	
	Let us point out some observations related to this result:
	\begin{itemize}
		\item In all the cases above, the interval Garside group $G([1, w_0]_T)$ is (directly) decomposable. When $w_0$ is replaced by a Coxeter element $c$, the interval group $G([1,c]_T)$ is the dual Artin group, isomorphic to the classical Artin group of type $W$, which as shown by Paris~\cite[Proposition 4.2]{Paris_sph} is always indecomposable when $W$ is irreducible. 
		\item The fact that $\mathbb{Z} \times F_k$ is a Garside group already appears in the seminal work of Dehornoy and Paris (it corresponds to the Garside structure given in~\cite[Example 5]{DP} with $p=n=m=k+1$); more recently, Haettel and Huang~\cite{HH} discovered Garside structures for many groups of the form $\mathbb{Z} \times G$. Their construction covers the case where $G$ is part of an important family of right-angled Artin groups.  The groups appearing in points $(1)$, $(3)$ and $(4)$ above are covered by their construction, and we obtain the same Garside structure as theirs. It is an open question to determine for which Artin groups $G$ the direct product $\mathbb{Z} \times G$ is a Garside group. 
	\end{itemize}

	In the last Section~\ref{sec:exple}, we give some examples of interval Garside groups similarly built from involutions or central elements in finite complex reflection groups. These examples show that the situation there can be very different from the real case, especially for central elements of order different from $2$, where the lattice property often fails. Lemmas~\ref{lem_latt_fail} and~\ref{lem_tool_failure}, which are formulated in the general setting of a group $G$ positively generated by a set $A$, capture some of the reasons explaining why the lattice property easily fails in this setting. In several irreducible complex reflection groups of rank $2$, taking $- I_2$ yields Garside structures already obtained in the dihedral case; we finish this section by treating the case $W=G(4,1,3)$ with the element $-I_3$, which gives a Garside group containing the one obtained for type $B_3$ with the interval $[1,w_0]_T$.

	\section{Prerequisites on reflection groups, interval groups and involutions}\label{sec:pre}
	
	\subsection{Interval Garside groups}\label{sub_interval}
	
	In this subsection, we recall the definition of Garside monoids and groups, and how to build interval Garside structures. For more on Garside theory we refer the reader to~\cite{DP, Garside}. 
	
	\begin{definition}[{\cite{DP}}]
		A \defn{Garside monoid} is a pair $(M,\Delta)$, where $M$ is a monoid and $\Delta\in M$, satisfying the following properties:
		\begin{enumerate}
			\item $M$ is left- and right-cancellative,
			\item $M$ is atomic,
			\item left- and right-divisibility endow $M$ with two lattice structures,
			\item the set of left-divisors of $\Delta$ coincides with its set of right-divisors, and generates $M$; this common set is denoted $\mathrm{Div}(\Delta)$,
			\item the set $\mathrm{Div}(\Delta)$ is finite.
		\end{enumerate}
		The elements of $\mathrm{Div}(\Delta)$ are called the \defn{simple elements} (or the \defn{simples}) of $M$.
	\end{definition}
	
	A Garside monoid satisfies Öre's conditions; hence $M$ embeds into a group of fractions $G_M$. Such a group is called a \defn{Garside group}: every monoid presentation of $M$, viewed as a group presentation, yields a presentation of $G_M$. Garside groups have a solvable word problem, are torsion free, have a nontrivial center, admit finite $K(\pi, 1)$'s, etc.--see~\cite{Garside, DP} or~\cite[Section 2]{Gobet_torus} for more on the topic. Seminal examples include (classical and dual) Artin groups of spherical type, and torus knot groups. 
	
	\medskip
	
	Proving that a given group $G$ is a Garside group is a hard task in general. We recall the basic theory of interval Garside groups, which is a powerful method for producing Garside groups. Let $G$ be a group and $A \subseteq G$ such that $G = \langle A \rangle^+$, i.e., such that $A$ positively generates $G$. We consider the length function $\ell_A$ on $G$ associated with the generating set $A$.  
	If $x \in G$ and $x = a_1 a_2 \cdots a_{\ell_A(x)}$ with $a_i \in A$, we say that the word $a_1 a_2 \cdots a_{\ell_A(x)}$ is an \defn{$A$-reduced expression} of $x$. Define a partial order $\leq_A$ on $G$ by:
	\[
	u \leq_A v \ \Leftrightarrow \ \ell_A(u) + \ell_A(u^{-1}v) = \ell_A(v).
	\]
	
	In other words, $u \leq_A v$ if and only if some $A$-reduced expression of $v$ has an $A$-reduced expression of $u$ as a prefix.
	
	Analogously, we define another partial order $\leq_{A, R}$ on $G$ by:
	\[
	u \leq_{A, R} v \ \Leftrightarrow \ \ell_A(u) + \ell_A(vu^{-1}) = \ell_A(v).
	\]
	
	\begin{definition}
		Let $G$ and $A$ be as above. An element $\delta \in G$ is said to be \defn{$A$-balanced} (or simply \defn{balanced} if $A$ is clear from the context) if
		$$\{ g \in G \mid g \leq_A \delta \} = \{ g \in G \mid g \leq_{A, R} \delta \}.$$
		
		We denote this set by $P_\delta$ or $[1, \delta]_A$.
	\end{definition}
	
	\begin{definition}
		Let $\delta \in G$ be an $A$-balanced element. The \defn{interval monoid} $M(P_\delta)$ is the monoid generated by a copy $\{\underline{u} \mid u \in P_\delta\}$ of $P_\delta$ and defined by the following presentation:
		\begin{equation}\label{eq_interval}
			M(P_\delta) = \langle \underline{u} \mid \underline{u} \cdot \underline{v} = \underline{w} \text{ if } u, v, w \in P_\delta, uv = w, \text{ and } u \leq_A w \rangle.
		\end{equation}
	\end{definition}
	
	Note that the map $\varphi: M(P_\delta) \to G$ defined by $\underline{u} \mapsto u$ is a monoid homomorphism. As a consequence, the set $\mathbf{P_\delta} = \{\underline{u} \mid u \in P_\delta\} \subseteq M(P_\delta)$ is in bijective correspondence with $P_\delta$.
	
	\begin{theorem}[{\cite[Theorem 0.5.2]{Bessis}}]
		Let $\delta \in G$ be $A$-balanced. If $(P_\delta, \leq_A)$ is a finite lattice, then $(M(P_\delta), \underline{\delta})$ is a Garside monoid. The associated Garside group $G_{M(P_\delta)}$ is simply denoted $G(P_\delta)$, and the set $\mathrm{Div}(\Delta)$ is given by $\mathbf{P_\delta}$. The map $(P_\delta, \leq_A) \to (\mathbf{P_\delta}, \leq), u \mapsto \underline{u}$ is an isomorphism of posets, where $\leq$ is the restriction of the left-divisibility order on $M(P_\delta)$.  
	\end{theorem}
	
	The above theorem can be used to show that both classical and dual Artin groups of spherical type are Garside groups; in both cases, the group $G$ is the attached Coxeter group $W$, and $A$ is respectively the set $S$ of simple reflections and the set $T$ of all reflections in $W$ (see Subsection~\ref{sub_coxeter} below for precise notation). The balanced element is the longest element $w_0$ in the first case, and a Coxeter element $c$ in the second case. 
	
	\subsection{Coxeter groups and absolute order}\label{sub_coxeter}
	
	We shall study interval groups built from Coxeter groups. Let $(W, S)$ be a Coxeter system and let $T = \bigcup_{w\in W} w S w^{-1}$ denote its set of reflections. We refer the reader to~\cite{Humphreys, BB} for background on Coxeter groups. Recall that a subgroup $W'$ of $W$ generated by a subset of $T$ is called a \defn{reflection subgroup} of $W$, and that it admits a canonical structure of Coxeter group, with set of reflections $T'= W' \cap T$ (see~\cite{Dyer_Ref}; the simple system $S'$ is harder to describe, and we shall not need it here). For $A\in \{S, T\}$, denote by $\ell_A: W \longrightarrow \mathbb{Z}_{\geq 0}$ the corresponding length function. Let $B_W$ denote the attached Artin group (obtained from the presentation of $W$ by removing the relations $s^2=1, s\in S$). When $W$ is finite, it can be realized as an interval group in two different ways. Assume that $W$ is finite and denote by $w_0$ its longest element, and by $c$ a Coxeter element (a product of the element of $S$ in some order). Reformulating the work of Brieskorn and Saito~\cite{BS}, extending and refining an original idea of Garside~\cite{Garside_69} in type $A$, one has that $B_W \cong G([1,w_0]_S)$. This in particular shows that Artin groups of spherical type are Garside groups, and gives and elegant, uniform proof of several important properties (solution to the word problem, $K(\pi, 1)$ conjecture, absence of torsion, etc.). The second realization is with $A=T$. Bessis~\cite{Bessis} indeed showed that $B_W \cong G([1,c]_T)$. One of the advantages of this last approach is that when $W$ is infinite, there is no longest element, while Coxeter elements still exist. In particular this approach also works for some infinite Coxeter groups (see~\cite{Bessis_free, Dig1, Dig2, MS, DPS}) and complex braid groups (see~\cite{Bessis_kp}). The approach which consists in studying $W$ and $B_W$ with the generating set $T$ rather than $S$ is usually called the "dual" approach. 
	
	\medskip
	What we shall do in the following sections is to "mix" the two approaches by considering finite Coxeter groups and elements $w_0$ but with a dual point of view, that is, with the generating set $T$. As explained in the introduction, a classification of the finite irreducible Coxeter groups for which $[1, w_0]_T$ is a lattice was given by the second author (see Theorem~\ref{thm_gob_class} and the paragraph after it). In the following section, we identify the various interval Garside groups arising from the finite Coxeter groups listed in Theorem~\ref{thm_gob_class}. 
	
	We list a few properties of the absolute order on intervals $[1,u]_T$, where $u$ is an involution, that we shall use later. Recall that a subgroup of $W$ generated by a subset $I$ of $S$ is called a \defn{standard parabolic subgroup} of $W$, and that every subgroup which is conjugate to a standard parabolic subgroup is a \defn{parabolic subgroup} of $W$. Parabolic subgroups are stable by intersection, and are examples of reflection subgroups. Given an element $w$ of a Coxeter group $W$, we denote by $P(w)$ its parabolic closure, that is, the intersection of all the parabolic subgroups containing $w$ (which is again parabolic). 
	
	\begin{rmq}\label{rmq:parab_clos}
		When $W$ is finite and $w\in W$, the set $P(w)\cap T$ of reflections of $P(w)$ turns out to be $\{t\in W \ \vert \ t \leq_T w\}$, and $P(w)$ has rank $\ell_T(w)$ as a reflection group (see for instance~\cite[Proposition 2.3]{Gobet_invol}). 
	\end{rmq}
	
	\begin{prop}[{\cite[Corollary 3.7, Propositions 3.10 and 3.12]{Gobet_invol}}]\label{prop:invols}
		Let $W$ be a finite Coxeter group. 
		\begin{enumerate}
			\item If $t_1 t_2 \cdots t_k$ is a $T$-reduced expression of an involution $u\in W$, then the $t_i$'s pairwise commute with each other.
			\item The following are equivalent:
			\begin{enumerate}
				\item We have $W = P(w_0)$,
				\item The element $w_0$ is central in $W$,
				\item We have $t\leq_T w_0$ for every $t\in T$.
			\end{enumerate}
			Moreover, if any of the above conditions holds, then as a set, the interval $[1, w_0]_T$ is the set of involutions of $W$. 
		\end{enumerate}
	\end{prop}
	
	\section{Identification of the interval Garside groups in the real case}\label{sec:id}
	
	In this section, we prove our main Theorem~\ref{thm:main}, by identifying the various interval Garside groups $G([1,w]_T)$ attached to those irreducible, finite Coxeter groups occurring in Theorem~\ref{thm_gob_class}. 
	
	\subsection{Even dihedral type}\label{sec_dih}
	Let $W$ be of type $I_2(2k)$ with $k\geq 2$, that is, $$W=\langle s_1, s_2 \ \vert \ s_1^2=s_2^2=1, \ (s_1 s_2)^k = (s_2 s_1)^k \rangle.$$ 
	We have $w_0=(s_1 s_2)^k = (s_2 s_1)^k$ and $\ell_T(w_0)=2$. There are $2k$ reflections in $W$, and by Proposition~\ref{prop:invols} every reflection $t\in T$ satisfies $t \leq_T w_0$. Since $\ell_T(w_0)=2$, for each $t\in T$ there is thus a unique $t'\in T$ such that $tt'= w_0$, and we moreover have $tt'=t't$ by Proposition~\ref{prop:invols} (or simply since $w_0$ is an involution). More explicitly, if $t=s_1 s_2 \cdots s_1$ where the number of factors is $\ell < 2k$, then $t'= s_2 s_1 \cdots s_2$, where the number of factors if $2k - \ell$. We label the reflections $t_1, t_2, \dots, t_{2k}$, in such a way that $t_i t_{i+1}= t_{i+1} t_i$ whenever $i$ is odd. The set $[1,w_0]_T$ is $T\cup \{1, w_0\}$ and the (nontrivial) defining relations of $M([1,w_0]_T)$ are then given by 
	\begin{align*}
		\underline{t_i} \cdot \underline{t_{i+1}} &= \underline{w_0}, \ \text{whenever~}i\text{~is odd}, \\
		\underline{t_{i+1}} \cdot \underline{t_{i}} &= \underline{w_0}, \ \text{whenever~}i\text{~is odd}.
	\end{align*}
	We thus get that $M([1,w_0]_T)$ (and hence $G([1,w_0]_T)$) admits the presentation
	\begin{equation}\label{gars_dih}
		\langle \ \underline{t_1}, \dots, \underline{t_{2k}} \ \vert \ \underline{t_1}\cdot \underline{t_2} = \underline{t_2} \cdot \underline{t_1} = \underline{t_3} \cdot \underline{t_4} = \underline{t_4} \cdot \underline{t_3} = \dots = \underline{t_{2k-1}} \cdot \underline{t_{2k}} = \underline{t_{2k}}\cdot \underline{t_{2k-1}} \ \rangle.
	\end{equation}
	Reintroducing the generator $\Delta:=\underline{w_0} = \underline{t_1} \cdot \underline{t_2}$, inside $G([1,w_0]_T)$ we can remove $\underline{t_{2\ell}}$ for all $\ell=1, \dots, k$ since $\underline{t_{2 \ell}} = \Delta \cdot\underline{t_{2\ell-1}}^{-1} = \underline{t_{2\ell -1}}^{-1}\cdot \Delta$. Relabeling the remaining $\underline{t_i}$'s by setting $r_\ell= \underline{t_{2\ell -1}}$ for all $\ell=1, \dots, k$, we get the presentation 
	\begin{equation}
		G([1,w_0]_T) = \langle \ r_1, r_2, \dots, r_k, \Delta  \ \vert \ r_i \Delta = \Delta r_i, \, \forall i = 1, \dots, k \ \rangle \cong F_k \times \mathbb{Z}.
	\end{equation}
	
	\begin{rmq}
		It is known that $F_k \times \mathbb{Z}$ is a Garside group; a Garside presentation already appeared in~\cite[Example 5]{DP}, and is given by \begin{equation*}
			\langle \ x_1, x_2, \dots, x_{k+1} \ \vert \ x_1 x_2 \cdots x_{k+1} = x_2 x_3 \cdots x_{k+1} x_1= \cdots = x_{k+1} x_1 x_2 \cdots x_k \ \rangle.
		\end{equation*}
		More recently, Haettel and Huang~\cite{HH} showed that for a large family of right angled Artin groups $G$ including $F_k$, the group $G \times \mathbb{Z}$ is a Garside group. The Garside presentation that they obtain for $F_k \times \mathbb{Z}$ is the same as the one given in~\eqref{gars_dih}.
	\end{rmq}
	
	\subsection{Type $B_n$}\label{sub_bn}
	Let $W$ be of type $B_n$, with simple system $S=\{s_0, s_1, \dots, s_{n-1}\}$, where $s_1, \dots, s_{n-1}$ are the simple transpositions generating the standard parabolic subgroup of type $A_{n-1}$, and $(s_0 s_1)^4=1$.
	
	We shall use the realization of $W$ as the group of signed permutations, that is, the group of permutations $\sigma$ of $[\pm n]:=\{\pm 1, \pm 2, \dots, \pm n \}$ such that $\sigma(-i)= - \sigma(i)$ for all $i\in [\pm n]$ (see~\cite[Section 8.1]{BB} for more details). We will work throughout with this combinatorial model.
	The reflections in $W$ are given by:
	\begin{itemize}[leftmargin=0.8cm]
		\item the elements $[i]:=(i,-i)$ for $i=1,\dots,n$, called \textit{balanced reflections};
		\item the elements $t_{i,j}=(i,j)(-i,-j)$ for $1\leq i<j\leq n$, called \textit{paired reflections of type I};
		\item the elements $r_{i,j}=(i,-j)(-i,j)$ for $1\leq i<j\leq n$, called \textit{paired reflections of type II}.
	\end{itemize}
	Notice that the elements $t_{i,j}$ are precisely the reflections of the standard parabolic subgroup of type $A_{n-1}$. We have $s_0=(1,-1)$ and $s_i=t_{i,i+1}$ for all $i=1,\dots,n-1$. The absolute length on $W$ was studied by Brady and Watt~\cite[Section 3]{BW}.
	
	The element $w_0$ is central in $W$, hence $W=P(w_0)$ by Proposition~\ref{prop:invols}. By~Theorem~\ref{thm_gob_class} and Proposition~\ref{prop:invols}, the interval $[1, w_0]_T$ is a lattice which coincides as a set with the set of involutions of $W$. This was already observed by Kallipoliti~\cite[Theorem 7 and Section 6]{Kal}.
	
	The involutions in $W$ may be described as follows: viewing $W$ inside the larger symmetric group $S_{2n}$ on $\{\pm 1, \pm 2, \dots, \pm n\}$, we get that every involution $u$ of $W$ is a product of transpositions of $S_{2n}$ with disjoint support. The condition $u(-i)=-u(i)$
	implies that the transpositions occurring are either of the form $(i, -i)$, or come in pairs $(i, j)$ and $(-i,-j)$, where $i \neq \pm j$. Therefore, $u$ is of the form \begin{equation}\label{form_ref_bn} u = \left(\prod_{i\in A} [i]\right)\left( \prod_{j=1}^k t_{a_j, b_j} \right)\left(\prod_{j=1}^{k'} r_{c_j, d_j}\right),\end{equation} where $A$ denotes the set of indices $i\in\{1,\dots,n\}$ such that $u(i)=-i$, and the sets
	$A$, $\{a_1,b_1\},\dots,\{a_k,b_k\}$, and $\{c_1,d_1\},\dots,\{c_{k'},d_{k'}\}$ are disjoint subsets of $\{1,\dots,n\}$.

	Conversely, every element of this form is an involution in $W$. By~\cite[Proposition 3.5]{BW}, we have $\ell_T(u)=|A|+k+k'$; in particular, the expression in~\eqref{form_ref_bn} is a $T$-reduced expression for $u$.
	Up to commutation of the factors, it yields a canonical
	$T$-reduced expression of $u$; in fact, this expression is closely related to the cycle decomposition of $u$ viewed inside $S_{2n}$: every factor appearing in~\eqref{form_ref_bn} corresponds either to a balanced cycle (for elements of the form $[i]$, $i\in A$), or to a symmetric pair of cycles (for the remaining ones) appearing in the cycle decomposition of $u$. 
	
	Using~\cite[Proposition 3.5]{BW} again, we see the following: given a balanced reflection $t=[i]$, we have $t \leq_T u$ if and only if $i\in A$. Given a paired reflection $t=t_{i,j}$  of type I (respectively a paired reflection $t=r_{i,j}$ of type II), we have $t \leq_T u$ if and only if either both $i,j$ lie in $A$, or there exists $p\in \{1,\dots,k\}$ (resp. $p\in \{1,\dots,k'\}$) such that $\{i,j\}=\{a_p,b_p\}$ (resp. $\{i,j\}=\{c_p,d_p\}$).

	\medskip

	Recall that the interval monoid $M([1, w_0]_T)$ is generated by a copy $\mathbf{P}_{w_0}= \{ \underline{u} \ \vert \ u \in W, u^2=1 \}$ of the set $P_{w_0}=[1,w_0]_T$ of involutions of $W$. Let $\underline{u}\in \mathbf{P}_{w_0}$, and suppose that $u= t_1 t_2 \cdots t_k$ is a $T$-reduced expression of $u$. Then $ t_1 t_2 \cdots t_i \leq_T u$ for all $i$, and repeated application of the defining relations yields $$ \underline{u} = \underline{t_1 \cdots t_{k-1}} \cdot \underline{t_k } = \underline{t_1 \cdots t_{k-2}} \cdot  \underline{t_{k-1}} \cdot  \underline{t_k} = \dots = \underline{t_1} \cdot\underline{t_2} \cdot \cdots \cdot \underline{t_k}.$$ Hence, the copy $\mathbf{T}= \{ \underline{t} \ \vert \ t\in T \}$ of the set of reflections generates $M([1, w_0]_T)$ since every reflection $t\in T$ satisfies $t\leq_T w_0$ (see Proposition~\ref{prop:invols}). 
	This is a special case of the general fact that, in the notation of Section~\ref{sub_interval}, the interval monoid $M(P_\delta)$ is generated by a copy of $P_\delta\cap A$, which in the present setting is $T$.
	
	Using this smaller generating set, the presentation~\eqref{eq_interval} can be rewritten as 
	\begin{equation}\label{brut_bn} M([1, w_0]_T) =
		\left\langle\, \mathbf{T} \ \middle| \ 
		\begin{array}{l}
			\underline{t_1} \cdots \underline{t_k} = \underline{q_1} \cdots \underline{q_k} \quad \text{whenever } \\
			t_1 \cdots t_k = q_1 \cdots q_k = u \in P_{w_0} \text{~and~} \ell_T(u) = k
		\end{array}
		\,\right\rangle,\end{equation}
	that is, lifts of $T$-reduced expressions of elements of $P_{w_0}$ are identified. We will show that it is sufficient to consider only relations arising from elements $u\in P_{w_0}$ with $\ell_T(u)=2$.

	\begin{lemma}\label{lem_rel_2_bn}
		With the above notation, we have 
		\begin{equation}\label{brut2_bn} M([1, w_0]_T) =
			\left\langle\, \mathbf{T} \ \middle| \ 
			\begin{array}{l}
				\underline{t_1} \cdot \underline{t_2} = \underline{q_1} \cdot \underline{q_2} \quad \text{whenever } \\
				t_1 \cdot t_2 = q_1 \cdot q_2 = u \in P_{w_0} \text{~and~} \ell_T(u) = 2
			\end{array}
			\,\right\rangle,\end{equation} which written in terms of balanced and paired reflections yields the presentation with generators $\underline{[i]}$, $\underline{t_{i,j}}$ and $\underline{r_{i,j}}$, and relations 
		\begin{align*}  \underline{[i]} \cdot \underline{[j]} = \underline{[j]} \cdot \underline{[i]} = \underline{t_{i,j}} \cdot \underline{r_{i,j}} = \underline{r_{i,j}} \cdot \underline{t_{i,j}}, & \quad \text{if } i < j, \\ \underline{[i]} \cdot \underline{t_{j,k}} = \underline{t_{j,k}} \cdot \underline{[i]}, \quad \underline{[i]} \cdot \underline{r_{j,k}} = \underline{r_{j,k}} \cdot \underline{[i]}, & \quad \text{if } |\{i,j,k\}| = 3, \ j < k ,\\  \underline{t_{i,j}} \cdot \underline{t_{k,\ell}} = \underline{t_{k,\ell}} \cdot \underline{t_{i,j}}, \quad \underline{t_{i,j}} \cdot \underline{r_{k,\ell}} = \underline{r_{k,\ell}} \cdot \underline{t_{i,j}}, & \\  \underline{r_{i,j}} \cdot \underline{r_{k,\ell}} = \underline{r_{k,\ell}} \cdot \underline{r_{i,j}}, & \quad \text{if } |\{i,j,k,\ell\}| = 4, \ i < j, \ k < \ell. \end{align*} 
	\end{lemma}
	
	\begin{proof}
		It suffices to show that any two $T$-reduced expressions $t_1 t_2 \cdots t_k$ and $q_1 q_2 \cdots q_k$ of an involution $u$ can be related by a sequence of applications of relations of the form $rr' = qq'$ on two successive factors, where $r,r',q,q'\in T$ and $rr'$ is an involution. We call such an application an \defn{admissible} move. Indeed, if this holds, then all defining relations appearing in~\eqref{brut_bn} follow from these, yielding Presentation~\eqref{brut2_bn}.
		
		We write $u$ in the form~\eqref{form_ref_bn}, that is,
		\begin{equation}\label{u_proof}u=
			\prod_{i\in A} [i]\;
			\prod_{j=1}^k t_{a_j,b_j}\;
			\prod_{j=1}^{k'} r_{c_j,d_j},
		\end{equation}
		where the subsets
		$A, \{a_1,b_1\}, \dots, \{a_k,b_k\}, \{c_1,d_1\}, \dots, \{c_{k'},d_{k'}\}$
		are disjoint subsets of $\{1,2,\dots,n\}$. It suffices to show that any $T$-reduced expression
		$t_1 t_2 \cdots t_{\ell_T(u)}$ of $u$ can be transformed into this canonical one by admissible moves.
		
		Recall from Proposition~\ref{prop:invols} that in such a $T$-reduced expression of an involution, the occurring factors $t_i$ pairwise commute. In particular, every adjacent product $t_i t_{i+1}$ is again an involution for all $i=1, \dots, \ell_T(u)-1$. Hence every commutation relation applied inside a $T$-reduced expression of an involution is an admissible move.

		We now claim that, up to commutation of the factors, every $T$-reduced expression of $u$ is of the form \begin{equation}\label{eq_invol_red_bn} [i_1] [i_2] \cdots [i_p] t_{e_1, f_1} r_{e_1, f_1} \cdots t_{e_\ell, f_\ell} r_{e_\ell, f_\ell} \prod_{j=1}^k t_{a_j, b_j} \prod_{j=1}^{k'} r_{c_j, d_j},\end{equation} where the sets $\{i_1, i_2, \dots, i_p\}$, $\{e_1, f_1\}$, \dots, $\{e_\ell, f_\ell\}$ form a partition of $A$.
		Firstly, observe that the above product equals $u$ since $[i_1] [i_2] \cdots [i_p] t_{e_1, f_1} r_{e_1, f_1} \cdots t_{e_\ell, f_\ell} r_{e_\ell, f_\ell}= \prod_{i\in A} [i]$. Moreover, it contains $p + 2\ell + k + k' = |A| + k + k' = \ell_T(u)$ factors, hence is $T$-reduced. Conversely, let $t_1 \cdots t_{\ell_T(u)}$ be a $T$-reduced expression of $u$, and let's show that, up to commutation, it is of the form~\eqref{eq_invol_red_bn}. By Proposition~\ref{prop:invols}, the reflections $t_i$ pairwise commute. Let
		$1 \leq j_1 < \cdots < j_p \leq \ell_T(u)$
		be the indices such that $t_{j_q}$ is balanced, and write $t_{j_q} = [i_q]$. Since $[i_q]$ has a support which is disjoint from the support of any distinct reflection commuting with it, we must have $u(i_q)=-i_q$, hence $i_q \in A$. Up to reordering, which amounts to apply commutation relations, we may assume that
		$t_1=[i_1],\dots,t_p=[i_p]$.

		Any remaining reflection $t$ among $t_{p+1}, \dots, t_{\ell_T(u)}$ is thus paired, with the additional requirements that 
		\begin{itemize}
			\item $\{i,j\} \cap \{i_1, i_2, \dots, i_p\} = \emptyset$ (otherwise there is a reflection in $t_1, t_2, \dots, t_p$ which does not commute with $t$),
			\item two distinct reflections $t, t'$ in $t_{p+1}, \dots, t_{\ell_T(u)}$ either have disjoint support, or up to commutation are of the form $t=t_{i,j}$, $t'=r_{i,j}$ (this last case is the only situation where two distinct reflections with nondisjoint supports can commute with each other). 
		\end{itemize}
		Assume that there are $t, t'$ in $\{t_{p+1}, \ldots, t_{\ell_T(u)}\}$ distinct but with nondisjoint supports. Then up to commutation $t= t_{i,j}$, $t'=r_{i,j}$, and since all the $t_i$'s are commuting with each other and $tt'=[i][j]$, then arguing as in the first case we see that $i, j\in A$. The reflections that come by pairs are thus of the form $t_{e_1, f_1}$, $r_{e_1, f_1}$, \dots, $t_{e_\ell, f_\ell}$, $r_{e_\ell, f_\ell}$ where the sets $\{e_1, f_1\}$, $\dots$, $\{e_\ell, f_\ell\}$ are included in $A$, disjoint, and disjoint from $\{i_1, \dots, i_p\}$. Now assume that $t$ is a reflection among $t_{p+1}$, $\dots$, $t_{\ell_T(u)}$ which has disjoint support with any other reflection among them. It also has disjoint support with $t_1, \dots, t_p$. It thus has disjoint support with every other reflection in $t_1, t_2, \dots, t_{\ell_T(u)}$. Since the latter are all pairwise commuting and $t$ is of the form $t_{i,j}$ or $r_{i,j}$, the only possibility to get a product equal to~\eqref{u_proof} is that $t$ is one of the $t_{a_j, b_j}$'s, or one of the $r_{c_j, d_j}$'s. 
		
		Putting everything together, we get that the $T$-reduced expression $t_1 t_2 \cdots t_{\ell_T(u)}$ is, up to commutation, of the form $$[i_1] [i_2] \cdots [i_p] t_{e_1, f_1} r_{e_1, f_1} \cdots t_{e_\ell, f_\ell} r_{e_\ell, f_\ell} \prod_{j=1}^k t_{a_j, b_j}^{\varepsilon_j} \prod_{j=1}^{k'} r_{c_j, d_j}^{\delta_j},$$ where the sets $\{i_1, i_2, \dots, i_p\}$, $\{e_1, f_1\}$, \dots, $\{e_\ell, f_\ell\}$ are disjoint and included in $A$, and $\varepsilon_j, \delta_j \in\{0, 1\}$. But the number of factors in the above expression is given by $p + 2\ell + \sum_{j=1}^k \varepsilon_j + \sum_{j=1}^{k'} \delta_j$, which has to be equal to $\ell_T(u)=|A| + k + k'$, which, since $p+2\ell \leq |A|$, forces $\varepsilon_j=1$ for all $j=1, \dots, k$, $\delta_j=1$ for all $j=1, \dots, k'$, and $p+2\ell = |A|$. The $T$-reduced expression thus has the required form up to commutation.

		The last statement is obtained by explicitly listing the relations corresponding to admissible moves. 
	\end{proof}
	
	For $n=3$, the lattice of simple elements, which is the same as the lattice of involutions in the Coxeter group of type $B_3$ partially ordered by the restriction of the absolute order, is given in Figure~\ref{fig:b3}. 
	
	\begin{figure}[h!]
		\[\begin{tikzcd}[
			scale=0.75,
			transform shape,
			column sep=small,
			row sep=small,     cells={nodes={font=\scriptsize}}
			]
			&&&& {[1][2][3]} &&&& \\
			{[1]t_{2,3}} & {[1][2]} & {[1] r_{2,3}} & {[2] t_{1,3}} & {[1][3]} & {[2]r_{1,3}} & {[3] t_{1,2}} & {[2][3]} & {[3] r_{1,2}} \\
			&&& {} \\
			{[1]} & {t_{1,2}} & {r_{1,2}} & {t_{1,3}} & {[2]} & {r_{1,3}} & {t_{2,3}} & {r_{2,3}} & {[3]} \\
			&&&& e
			\arrow[no head, from=2-1, to=1-5]
			\arrow[no head, from=2-2, to=1-5]
			\arrow[no head, from=2-3, to=1-5]
			\arrow[no head, from=2-4, to=1-5]
			\arrow[no head, from=2-5, to=1-5]
			\arrow[no head, from=2-6, to=1-5]
			\arrow[no head, from=2-7, to=1-5]
			\arrow[no head, from=2-8, to=1-5]
			\arrow[no head, from=2-9, to=1-5]
			\arrow[no head, from=4-1, to=2-1]
			\arrow[no head, from=4-1, to=2-2]
			\arrow[no head, from=4-1, to=2-3]
			\arrow[no head, from=4-1, to=2-5]
			\arrow[no head, from=4-2, to=2-2]
			\arrow[no head, from=4-2, to=2-7]
			\arrow[no head, from=4-3, to=2-2]
			\arrow[no head, from=4-3, to=2-9]
			\arrow[no head, from=4-4, to=2-4]
			\arrow[no head, from=4-4, to=2-5]
			\arrow[no head, from=4-5, to=2-2]
			\arrow[no head, from=4-5, to=2-4]
			\arrow[no head, from=4-5, to=2-6]
			\arrow[no head, from=4-5, to=2-8]
			\arrow[no head, from=4-6, to=2-5]
			\arrow[no head, from=4-6, to=2-6]
			\arrow[no head, from=4-7, to=2-1]
			\arrow[no head, from=4-7, to=2-8]
			\arrow[no head, from=4-8, to=2-3]
			\arrow[no head, from=4-8, to=2-8]
			\arrow[no head, from=4-9, to=2-5]
			\arrow[no head, from=4-9, to=2-7]
			\arrow[no head, from=4-9, to=2-8]
			\arrow[no head, from=4-9, to=2-9]
			\arrow[no head, from=5-5, to=4-1]
			\arrow[no head, from=5-5, to=4-2]
			\arrow[no head, from=5-5, to=4-3]
			\arrow[no head, from=5-5, to=4-4]
			\arrow[no head, from=5-5, to=4-5]
			\arrow[no head, from=5-5, to=4-6]
			\arrow[no head, from=5-5, to=4-7]
			\arrow[no head, from=5-5, to=4-8]
			\arrow[no head, from=5-5, to=4-9]
		\end{tikzcd}\]
		\caption{The lattice $[1,w_0]_T$ in type $B_3$ which, as a set, is the set of involutions of $B_3$. It is the lattice of simples of $M([1,w_0]_T)$ in type $B_3$.}
		\label{fig:b3}
	\end{figure}
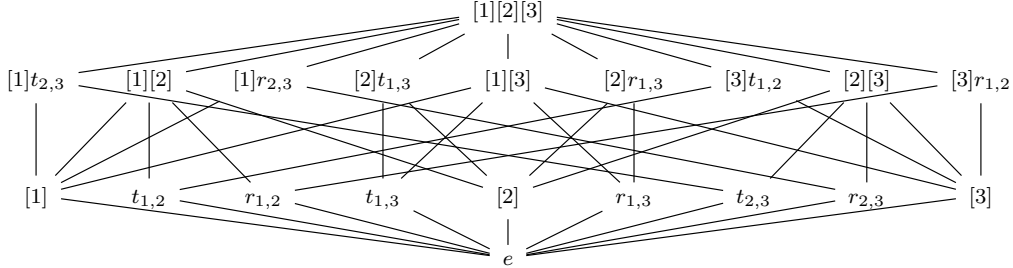
	
	\begin{rmq}
		Lemma~\ref{lem_rel_2_bn} tells us that defining relations with both sides consisting of words of length $2$ are enough to generate $M([1,w_0]_T)$. For dual braid monoids attached to spherical type Coxeter groups, this property also holds (see~\cite{Bessis}). At the level of Coxeter groups, this can be seen as a kind of dual version of the Mastumoto property relating $S$-reduced expressions of elements of Coxeter groups. Nevertheless, as we shall see in type $D_4$ below (see Subsection~\ref{d4}), in the case of interval monoids attached to involutions it does not hold all the time, hence using the specific combinatorics of the Coxeter groups of type $B_n$ rather than a general argument seems necessary. Also, for Coxeter elements a classical way to formulate this property for $T$-reduced expressions is to use the so-called \textit{Hurwitz action} of the braid group; the fact that any two $T$-reduced expressions can be related by "local" moves is equivalent to the transitivity of this action. Here, even if this property of connectivity by "local" moves stays valid, the Hurwitz action is irrelevant: as we have seen, we have local relations of the form $[i] [j] = t_{i,j} r_{i,j}$, hence with $4$ different factors, while when considering the Hurwitz action there is always one common factor in each side. 
	\end{rmq}
	
	We now study the interval group $G([1,w_0]_T)$. The presentations from Lemma~\ref{lem_rel_2_bn} have a minimal number of generators as monoid presentations, but not as group presentations. We shall remove some of the generators to get a more tractable presentation:

	\begin{prop}\label{prop_bn_clean}
		The group $G([1, w_0]_T)$ is isomorphic to $$ \mathbb{Z} \times \left\langle\, \begin{array}{l} [1], \dots, [n-1], \\ t_{i,j}, \quad 1 \leq i < j \leq n \end{array} \ \middle| \ \begin{array}{ll} \underline{[i]} \cdot \underline{[j]} = \underline{[j]} \cdot \underline{[i]}   & \text{if } i < j, \\[4pt] \underline{[i]} \cdot \underline{[j]} \cdot\underline{t_{i,j}} =\underline{t_{i,j}} \cdot \underline{[i]} \cdot \underline{[j]} & \text{if } i < j,\\[4pt] \underline{[i]} \cdot \underline{t_{j,k}} = \underline{t_{j,k}} \cdot \underline{[i]}  & \text{if } |\{i,j,k\}| = 3, \ j < k \\[4pt] \underline{t_{i,j}} \cdot \underline{t_{k,\ell}} = \underline{t_{k,\ell}} \cdot \underline{t_{i,j}}  & \text{if } |\{i,j,k,\ell\}| = 4, \ i < j, \ k < \ell. \end{array} \,\right\rangle $$
	\end{prop}
	
	\begin{proof}
		The presentation given in~\eqref{brut2_bn} is also a presentation of $G([1, w_0]_T)$. Removing the generators $\underline{r_{i,j}}$ via $\underline{r_{i,j}}= \underline{[i]} \cdot \underline{[j]} \cdot \underline{t_{i,j}}^{-1}$ yields a presentation which simplifies to \[\left\langle\, \begin{array}{l} [1], \dots, [n], \\ t_{i,j}, \quad 1 \leq i < j \leq n \end{array} \ \middle| \ \begin{array}{ll} \underline{[i]} \cdot \underline{[j]} = \underline{[j]} \cdot \underline{[i]}   & \text{if } i < j, \\[4pt] \underline{[i]} \cdot \underline{[j]} \cdot\underline{t_{i,j}} =\underline{t_{i,j}} \cdot \underline{[i]} \cdot \underline{[j]} & \text{if } i < j,\\[4pt] \underline{[i]} \cdot \underline{t_{j,k}} = \underline{t_{j,k}} \cdot \underline{[i]}  & \text{if } |\{i,j,k\}| = 3, \ j < k \\[4pt] \underline{t_{i,j}} \cdot \underline{t_{k,\ell}} = \underline{t_{k,\ell}} \cdot \underline{t_{i,j}}  & \text{if } |\{i,j,k,\ell\}| = 4, \ i < j, \ k < \ell. \end{array} \,\right\rangle\]
		Note that $\Delta:= \underline{[1]} \cdot \underline{[2]} \cdots \underline{[n]}$, which is in fact $\underline{w_0}$, is central. This follows by definition of the interval group as $w_0$ is central in $B_n$ hence commutes with every reflection in $B_n$, but it can also be checked easily using the above presentation. We can thus add it as a generator, and replace $\underline{[n]}$ by $\underline{[n-1]}^{-1} \cdots \underline{[2]}^{-1} \cdot \underline{[1]}^{-1} \cdot \Delta$. This replacement only adds the relations $\Delta x = x \Delta$ for every other generator $x$, all other added relations are redundant. One can thus factor a copy of $\mathbb{Z}$ generated by $\Delta$ and obtain the claimed presentation. 
	\end{proof}
	
	\begin{exple}\label{ex_bn}
		For $n=3$, setting $a=\underline{[1]}, b=\underline{[2]}, d=\underline{t_{1,2}}, c=\underline{t_{1,3}}, e=\underline{t_{2,3}}$ the group in Proposition~\ref{prop_bn_clean} becomes 	\[\mathbb{Z} \times \Biggl\langle 
		\begin{array}{l|cl}
			& ab=ba, \\
			a,b,c,d,e & cb=bc,  & abd=dab \\
			& ae=ea,  &                                             
		\end{array}
		\Biggr\rangle.\]
	\end{exple}
	
	\subsection{Types $D_4$ and $H_3$}

	\subsubsection{Type $D_4$}\label{d4}
	
	Recall that the Coxeter group of type $D_n$ can be realized as a reflection subgroup of index $2$ of $B_n$. The combinatorics introduced in Section~\ref{sub_bn} can thus be used for type $D_n$ as well, and we will keep the notation introduced there. Moreover, when a Coxeter group $W$ is finite, the reflection length in a reflection subgroup $W'$ is equal to the reflection length in $W$; that is, denoting $T$ (respectively $T' = W' \cap T$) the set of reflections in $W$ (respectively $W'$), for all $w\in W'$ we have $\ell_{T'}(w)=\ell_T(w)$ (see for instance~\cite[Lemma~3.10]{Gobet_cycle}). 
	
	In the notation of Section~\ref{sub_bn}, the longest element of $D_4$ is equal to the longest element $w_0$ of $B_4$, given by $[1][2][3][4]$, and since $w_0$ is central, by Proposition~\ref{prop:invols} the parabolic closure of $w_0$ is $D_4$. But there is a significant difference with the situation in type $B_n$: in the proof of Lemma~\ref{lem_rel_2_bn}, we showed that if $W$ is of type $B_n$ and $w_0$ is its longest element, then every two $T$-reduced expressions of $u \in [1, w_0]_T$ can be related by a sequence of application of relations where only two successive letters are modified at each step, leading to Presentation~\eqref{brut2_bn}. This does not hold in $D_n$ in general. Indeed, in $D_4$ the set of reflections is given by $\{t_{i,j} \ \vert \ 1 \leq i < j \leq 4\} \cup \{ r_{i,j} \ \vert \ 1 \leq i < j \leq 4\}$. By Proposition~\ref{prop:invols}, the interval $[1, w_0]_T$ in $D_4$ consists of all involutions in $D_4$. The involutions in $D_4$ which have reflection length equal to $2$ are either of the form $t_{i,j} r_{i,j}$ or $u_{i,j} v_{k, \ell}$ with $|\{i,j,k,l\}|=4$ and $u, v\in\{t,r\}$, and unlike in type $B_n$, the only relations arising from those elements are commutation relations. The two $T$-reduced expressions of $w_0$ $$ t_{1, 2} r_{1, 2} t_{3, 4} r_{3, 4} \ \text{and} \ t_{1, 3} r_{1, 3} t_{2, 4} r_{2, 4}$$ cannot be related using only commutation relations since they are made up of different factors. As a consequence, relations with both sides of length $4$ will persist in the presentation of $M([1,w_0]_T)$. We have:
	
	\begin{lemma}\label{lem_d4_1}
		The monoid $M([1, w_0]_T)$ admits the presentation
		\begin{equation*}
			\left\langle
			\begin{array}{c}
				\underline{r_{i,j}},\,\underline{t_{i,j}} \\[2pt]
				1\le i<j\le 4
			\end{array}
			\ \middle|\
			\begin{array}{l}
				\underline{t_{i,j}}\cdot\underline{r_{i,j}}
				= \underline{r_{i,j}}\cdot \underline{t_{i,j}}, \quad \forall 1 \leq i < j \leq 4,\\[6pt]
				\underline{u_{i,j}}\cdot\underline{v_{k,\ell}}
				= \underline{v_{k,\ell}}\cdot\underline{u_{i,j}},\quad \text{if } u, v\in \{t, r\}\text{~and~} |\{i,j,k,\ell\}|=4, \\[6pt]
				\underline{t_{i,j}}\cdot\underline{r_{i,j}}\cdot
				\underline{t_{k,\ell}}\cdot\underline{r_{k,\ell}}
				=\underline{t_{a,b}}\cdot\underline{r_{a,b}}\cdot
				\underline{t_{c,d}}\cdot\underline{r_{c,d}},\quad \text{if } 
				|\{i,j,k,\ell\}|=|\{a,b,c,d\}|=4.
			\end{array}
			\right\rangle .
		\end{equation*}
	\end{lemma}
	\begin{proof}
		Since $D_4$ is a reflection subgroup of $B_4$ and the reflection length stays unchanged in a reflection subgroup, we have $[1, w_0]_{T'} \subseteq [1, w_0]_T$; here $T'$ (respectively $T$) denotes the set of reflections of $D_4$ (resp. of $B_4$). We can reuse some facts established in the proof of Lemma~\ref{lem_rel_2_bn}. We proved that every $T$-reduced expression of an involution in $B_n$ is of the form~\eqref{eq_invol_red_bn}, and those which are reduced $T$-expressions of elements of $D_n$ with every factor in $D_n$ (that is, $T'$-reduced expressions) are those for which $p=0$. Note that the element of $B_n$ represented by the word~\eqref{eq_invol_red_bn} is given by $$ \prod_{i\in A} [i] \prod_{j=1}^k (a_j, b_j)(-a_j, -b_j) \prod_{j=1}^{k'} (c_j, -d_j)(-c_j, d_j).$$
		
		As in type $B_n$, we can work out a presentation of $M([1, w_0]_T)$ of the form~\eqref{brut_bn}, obtained by listing all the elements of $P_{w_0}$ and for every such element, equating all its lifted $T'$-reduced expressions. 
		
		Involutions of reflection length $2$ are, as already observed, those with a $T'$-reduced expression of the form $t_{i,j} r_{i,j}$ or $u_{i,j} v_{k,\ell}$, with $\{i,j,k,\ell\}=\{1,2,3,4\}$ and $u, v\in \{t, r\}$. Such elements only have one other $T'$-reduced expression, obtained by exchanging the two factors of the one given above. Hence elements of reflection length $2$ only produce commutation relations in the interval monoid, and those correspond to the relations in the first two lines of the presentation above.
		
		Involutions of reflection length $3$ are, since $n=4$, necessarily of the form $[i] [j] (k, \pm \ell )(- k, \mp \ell )$, with $\{i,j,k,\ell\}=\{1, 2, 3, 4\}$. Every $T'$-reduced expression of such an element is necessarily, up to commutation, of the form $t_{i,j} r_{i,j} u_{k, \ell}$. All the commutation relations required already arise from elements of reflection length $2$, hence involutions of reflection length $3$ do not contribute any new relation. 
		
		Finally, the only involution of reflection length $4$ is $w_0 = [1][2][3][4]$. Every $T'$-reduced expression of such an element is, up to commutation, of the form $t_{i,j} r_{i,j} t_{k,\ell} r_{k,\ell}$, with $\{i,j,k,\ell\} = \{1, 2, 3, 4\}$. The commutation relations already arose from elements of reflection length $2$, hence the only relations to add are those equating two $T'$-reduced expressions of $w_0$ as above, yielding the last line of relations. 
	\end{proof}
	
	\begin{cor}
		The group $G([1, w_0]_T)$ is isomorphic to $\mathbb{Z} \times \mathrm{RAAG}(\triangle ~ \triangle ~ \triangle)$.
	\end{cor}
	
	\begin{proof}
		The group $G([1, w_0]_T)$ has the same presentation as $M([1, w_0]_T)$. We start from the presentation obtained in Lemma~\ref{lem_d4_1}. Setting $\Delta:=\underline{t_{i,j}}\cdot\underline{r_{i,j}}\cdot
		\underline{t_{k,\ell}}\cdot\underline{r_{k,\ell}}$ (where $\{i,j,k,\ell\}=\{1,2,3,4\}$), we get that $\Delta$ is central (this also follows from the fact that it is equal to $\underline{w_0}$, that $t w_0 = w_0 t$ for every $t\in T'$ by Proposition~\ref{prop:invols}, and by construction of the interval monoid). Since $\Delta = \underline{t_{i,j}}\cdot\underline{r_{i,j}}\cdot
		\underline{t_{k,\ell}}\cdot\underline{r_{k,\ell}}$ for \textit{every} choice of $i,j,k,\ell$ such that $\{i,j,k,\ell\}=\{1,2,3,4\}$, by adding $\Delta$ as generator we can get rid of three of the six $r_{i,j}$'s, say $\underline{r_{2, 3}}, \underline{r_{2, 4}}$ and $\underline{r_{3, 4}}$ as $ \underline{r_{k, \ell}} = \underline{t_{k,\ell}}^{-1} \underline{r_{i,j}}^{-1}\underline{t_{i,j}}^{-1} \Delta$. One checks that adding $\Delta$ as generator and removing $\underline{r_{2, 3}}, \underline{r_{2, 4}}$ and $\underline{r_{3, 4}}$ yields a presentation which, after simplification, removes every relation in which one of the generators $\underline{r_{2, 3}}, \underline{r_{2, 4}}$ and $\underline{r_{3, 4}}$ is involved and just adds the relation $\Delta a = a \Delta$ for every other generator $a$. This new presentation has relations 
		\begin{align*}
			\underline{t_{1, 2}} \cdot \underline{t_{3, 4}}= \underline{t_{3, 4}} \cdot \underline{t_{1, 2}}, \ \underline{t_{1, 2}} \cdot \underline{r_{1, 2}} = \underline{r_{1, 2}} \cdot \underline{t_{1, 2}}, \ \underline{t_{3, 4}} \cdot \underline{r_{1, 2}} = \underline{r_{1, 2}} \cdot \underline{t_{3, 4}}, \\
			\underline{t_{1, 3}} \cdot \underline{t_{2, 4}} = \underline{t_{2, 4}} \cdot \underline{t_{1, 3}}, \ \underline{t_{1, 3}} \cdot \underline{r_{1, 3}} = \underline{r_{1, 3}} \cdot \underline{t_{1, 3}}, \ \underline{t_{2, 4}} \cdot \underline{r_{1, 3}} = \underline{r_{1, 3}} \cdot \underline{t_{2, 4}}, \\
			\underline{t_{1, 4}} \cdot \underline{t_{2, 3}} = \underline{t_{2, 3}} \cdot \underline{t_{1, 4}}, \ \underline{t_{1, 4}} \cdot \underline{r_{1, 4}} = \underline{r_{1, 4}} \cdot \underline{t_{1, 4}}, \ \underline{t_{2, 3}} \cdot \underline{r_{1, 4}} = \underline{r_{1, 4}} \cdot \underline{t_{2, 3}}, \\
			\Delta a = a \Delta, \ \forall a\in \{ \underline{t_{1, 2}}, \underline{t_{1, 3}}, \underline{t_{1, 4}}, \underline{t_{2, 3}}, \underline{t_{2, 4}}, \underline{t_{3, 4}}, \underline{r_{1, 2}}, \underline{r_{1, 3}}, \underline{r_{1, 4}}\}.       
		\end{align*}
		This exactly yields the claimed presentation, with $\Delta$ as generator of $\mathbb{Z}$ and each line above except the last one corresponding to a triangle. 
	\end{proof}
	
	\subsubsection{Type $H_3$}
	
	In type $H_3$ the longest element $w_0$ is central and has reflection length equal to $3$. Note that in this case, the lattice property of $[1,w_0]_T$, which follows from Theorem~\ref{thm_gob_class}, is also a consequence of~\cite[Theorem 2.2]{Gob_max}. We thus have by Proposition~\ref{prop:invols} that the parabolic closure of $w_0$ is the whole group $H_3$ and by Proposition~\ref{prop:invols} we have $t \leq_T w_0$ for every reflection $t$. There are $15$ reflections in $H_3$. 
	
	Let $w = t_1 t_2$ be an involution of reflection length $2$, where $t_i\in T$. We have $t_1 t_2 = t_2 t_1$, and by Remark~\ref{rmq:parab_clos}, the parabolic closure of $t_1t_2$ is necessarily of type $A_1 \times A_1$, since it contains $t_1$ and $t_2$ (because $t_i \leq_T t_1 t_2$) and $A_1 \times A_1$ is the only class of rank two parabolic subgroups of $H_3$ which admit distinct and commuting reflections (the others rank two parabolic subgroups are of type $A_2$ or $I_2(5)$). It follows that the only relations arising from elements of reflection length $2$ are commutation relations: if $t_1 t_2$ was equal to $t_3 t_4$ with $t_1, t_2, t_3, t_4$ all distinct, then they would lie in the same rank $2$ parabolic subgroup (the parabolic closure of $t_1 t_2$), which is impossible. 
	
	Given $t_1\in T$, and letting $x\in W$ such that $t_1 x = w_0$, we then have that $x$ is an involution of reflection length $2$, because by Proposition~\ref{prop:invols} we have $t\leq_T w_0$ for every $t\in T$ and every element below $w_0$ for $\leq_T$ is an involution. There are thus $t_2, t_3\in T$ such that $w_0 = t_1 t_2 t_3$, and by what we observed in the previous paragraph, the pair $\{t_2, t_3\}$ is uniquely determined by $x$. Moreover, by Proposition~\ref{prop:invols} every $T$-reduced expression of $w_0$ is made of pairwise commuting reflections. Up to commutation of the letters, there are thus five $T$-reduced expressions of $w_0$, since $|T|=15$. We shall write $t_1, t_2, \dots, t_{15}$ the reflections of $W$, in such a way that for all $ 0 \leq k \leq 4$, we have $w_0 = t_{1+3k} t_{2+3k} t_{3+3k}$. We thus get the following: 
	
	\begin{lemma}\label{h3_brut} 
		Setting $I=\{\{1, 2, 3\}, \{4, 5, 6\}, \{7,8,9\}, \{10, 11, 12\}, \{13, 14, 15\}\},$ the monoid $M([1, w_0]_T)$ admits the presentation 
		\begin{equation*}
			\left\langle
			\begin{array}{c}
				\underline{t_{i}},\ i=1, \dots, 15
			\end{array}
			\ \middle|\
			\begin{array}{l}
				\underline{t_{i}}\cdot\underline{t_{j}}
				= \underline{t_{j}}\cdot \underline{t_{i}}, \quad \text{if }\exists S \in I \ \text{such that} \ i, j\in S,\\[6pt]
				\underline{t_{i}}\cdot\underline{t_{j}} \cdot\underline{t_k}
				= \underline{t_{\ell}}\cdot \underline{t_{p}} \cdot \underline{t_q}, \quad \text{if } \{i,j,k\} \in I, \{\ell, p, q\} \in I.
			\end{array}
			\right\rangle .
		\end{equation*}
		
	\end{lemma}
	
	As in the previously studied cases $B_n$ and $D_4$, when passing to the group of fractions we can remove some of the generators. 
	
	\begin{cor}
		The group $G([1, w_0]_T)$ is isomorphic to $\mathbb{Z} \times \mathrm{RAAG}( ~| ~| ~ | ~ | ~ |~)$.
	\end{cor}
	
	\begin{proof}
		The group $G([1, w_0]_T)$ has the same presentation as $M([1, w_0]_T)$. Starting from the presentation obtained in Lemma~\ref{h3_brut}, if we add the generator $\Delta = \underline{t_1} \cdot \underline{t_2} \cdot \underline{t_3}$, we can get rid of the generators $\underline{t_i}$ for $i = 3 + 3k$, where $0 \leq k \leq 4$, since $\underline{t_i} = \underline{t_{i-1}}^{-1} \cdot \underline{t_{i-2}}^{-1} \cdot \Delta$. Removing them and adding $\Delta$, simplifying the presentation and setting $J = \{1, 4, 7, 10, 13\}$, we get the presentation 
		\begin{equation*}
			\left\langle
			\begin{array}{c}
				\underline{t_i}, \underline{t_{i+1}}, \ i \in J
			\end{array}
			\ \middle|\
			\begin{array}{l}
				\Delta \cdot \underline{t_i} = \underline{t_i } \cdot \Delta, \ \forall i \in J \cup (J+1),\\[6pt]
				\underline{t_{i}}\cdot\underline{t_{i+1}}= \underline{t_{i+1}}\cdot \underline{t_{i}} , \quad \forall i \in J.
			\end{array}
			\right\rangle ,
		\end{equation*}
		which is isomorphic to $\mathbb{Z} \times \mathrm{RAAG}( ~| ~| ~ | ~ | ~ |~)$. 
	\end{proof}

	\section{Some examples and counter-examples in complex reflection groups}\label{sec:exple}
	
	\subsection{Some general observations}
	
	It is natural to wonder whether some of the previous results admit generalizations for complex reflection groups (we refer the reader to~\cite{LT} for an introduction to such groups). Recall that a (finite) complex reflection group is a finite subgroup of $\mathrm{GL}_n(\mathbb{C})$ generated by complex reflections, that is, elements of finite order whose set of fixed points is a hyperplane of $\mathbb{C}^n$. Let $W$ be a complex reflection group and denote $\mathrm{Ref}(W)$ its set of reflections. Note that it is a priori unclear by what to replace $w_0$ in this setting, since there is no notion of "longest element" in such groups. Nevertheless, in those cases studied in the previous section, by Proposition~\ref{prop:invols} the element $w_0$ is also an element of the center of $W$. Finite complex reflection groups often admit a nontrivial center, hence one can replace $w_0$ by a generator of the center (or an arbitrary nontrivial central element). Note that such elements are not involutive in general in complex reflection groups. But by Steinberg's theorem, any involution in a complex reflection group is always central in a parabolic subgroup: indeed, given $W$ a finite complex reflection group acting on $V$ and $w\in W$ an involution, we have $w \in C_W(V^w)$, which is a parabolic subgroup of $W$ acting as a reflection group on $(V^w)^\perp$, on which $w$ acts by $-\mathrm{id}$.
	
	Also, reflections may have order greater than two, which, as we shall see, will lead to very different behaviours than in the real case. In such cases, instead of taking the whole set of reflections $\mathrm{Ref}(W)$ as set of generators, one can also take the \textit{distinguished} ones; if $r\in \mathrm{Ref}(W)$ is a reflection, then denoting $H_r$ its reflecting hyperplane, the subgroup $C_W(H_r)$ is a cyclic reflection subgroup, generated by a single distinguished reflection $r'$ having $e^{2i \pi / k}$ where $k=|C_W(H_r)|$ as nontrivial eigenvalue. This subset of the set of reflections still generates the group, and is stable by conjugation. 
	
	The following lemmas will be helpful to show that in many cases, for a complex reflection group $W$ with central element $z$, the poset $[1, z]_A$ fails to be a lattice, where $A$ is either the set $\mathrm{Ref}(W)$ or the subset of distinguished reflections. They can be formulated in a quite general setting: 
	
	\begin{lemma}\label{lem_latt_fail}
		Let $G$ be a group positively generated by a set $A$ of elements which is stable by conjugation. Let $z\in Z(G)$. Assume that there are $s, t\in A$ such that 
		\begin{itemize}
			\item $st \notin A$, 
			\item $st \neq ts$,
			\item $st \leq_A z$.
		\end{itemize}
		Then $\ell_A(z) \geq 3$, and the poset $[1, z]_A$ fails to be a lattice. 
	\end{lemma}
	
	In practice, as we shall see below, when $z$ is an element of a complex reflection group which fails to be of order $2$, it is often not difficult to find a pair $s, t$ of generators satisfying the assumptions of the lemma. 
	
	\begin{proof}
		Since $st\notin A$, we have $\ell_A(st)=2$. The condition $st \leq_A z$ ensures that $\ell_A(z) \geq 2$. If $\ell_A(z)=2$, then we have $z = st$. Since $z$ is central, this forces $z= ts$, contradicting $st \neq ts$. Hence $\ell_A(z) \geq 3$. This establishes the first claim. 
		
		Since $st \leq_A z$, there are $t_3, t_4, \dots, t_{\ell_A(z)} \in A$ such that $z = s t t_3 t_4 \cdots t_{\ell_A(z)}$. Using that $z$ is central, we can write $$ z = t t_3 t_4 \cdots t_{\ell_A(z)} s = t s (s^{-1} t_3 s) (s^{-1} t_4 s) \cdots (s^{-1} t_{\ell_A(z)} s), $$ and since $A$ is stable by conjugation, the last expression is still an $A$-reduced expression of $z$. We conclude that $ts \leq_A z$. Since $st = t (t^{-1} s t)$, we have $t \leq_A st$, $t\leq_A ts$, and similarly $s \leq_A st$, $s \leq_A ts$. This gives rise to a so-called "bowtie" in the poset $[1, z]_A$ (see Figure~\ref{fig_bowtie}; for more on bowties see~\cite{BmC}), hence the lattice property fails: $s$ and $t$ fail to have a join. 
		\begin{figure}
			\[\begin{tikzcd}
				st && ts \\
				s && t
				\arrow[no head, from=2-1, to=1-1]
				\arrow[no head, from=2-1, to=1-3]
				\arrow[no head, from=2-3, to=1-1]
				\arrow[no head, from=2-3, to=1-3]
			\end{tikzcd}\]
			\caption{Bowtie as appearing in the proof of Lemma~\ref{lem_latt_fail}.}
			\label{fig_bowtie}
		\end{figure}
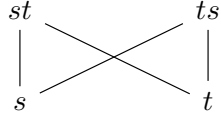
	\end{proof}
	
	\begin{lemma}\label{lem_tool_failure}
		Let $G$ be a group positively generated by a set $A$ of elements which is stable by conjugation. Let $z\in Z(G)$. Assume that the elements of $A$ all have the same finite order $N$. Let $z\in Z(G)$. If the order of $z$ does not divide $N$, then $[1,z]_A$ is not a lattice. 
	\end{lemma}
	
	\begin{proof}
		Let $t_1 t_2 \cdots t_{\ell_A(z)}$ be an $A$-reduced expression of $z$. If $t_i t_j = t_j t_i$ for all $i, j$, then $z$ has order dividing $N$, which is impossible. Hence there are $i, j$, $i<j$ such that $t_i t_j \neq t_j t_i$. Moving $t_i$ and $t_j$ to the left of the $A$-reduced expression (as we did for $s$ in the proof of Lemma~\ref{lem_latt_fail}), we get that $t_i t_j \leq_A z$. Since $t_i t_j \notin A$, the assumptions of Lemma~\ref{lem_latt_fail} are fulfilled with $s= t_i$ and $t=t_j$. 
	\end{proof}

	\subsection{Examples and counter-examples}
	
	We provide examples and counter-examples to the lattice property of $[1, z]_A$ for some complex reflection groups $W$ with $z\in Z(W)$ and $A \subseteq \mathrm{Ref}(W)$ stable by conjugation. Many properties used can be checked by computer (for instance using GAP) but when short explanations are possible without recourse to the computer we chose give them. When we obtain a lattice we also describe the attached interval Garside group. 
	
	It turns out that Lemma~\ref{lem_tool_failure} allows one to rule out many cases where the lattice property fails among irreducible complex reflection groups of rank $2$: 
	
	\begin{prop}
		Let $W$ be an irreducible complex reflection group of type $G_4$, $G_5$, $G_{13}$, $G_{16}$, $G_{20}$, $G_{22}$, and $z$ be a generator of $Z(W)$. Let $A$ denote either the set of reflections, or the set of distinguished reflections. Then $[1, z]_A$ is not a lattice. 
	\end{prop}
	
	\begin{proof}
		For these groups, the assumptions of Lemma~\ref{lem_tool_failure} are fulfilled. The orders of $z$ and of the reflections are collected in Table~\ref{table:order}, and can be found in~\cite[Chapter 6]{LT}. Note that the center of an irreducible complex reflection group is cyclic (it consists of scalar matrices). 
		\begin{table}[h!]
			\centering
			\begin{tabular}{|c|c|c|}
				\hline
				Group $W$ & Order of $z$ such that $\langle z \rangle = Z(W)$  & Order of the reflections \\ \hline
				$G_4$ & $2$ & $3$ \\ \hline
				$G_5$ & $6$ & $3$ \\ \hline
				$G_{13}$ & $4$ & $2$ \\ \hline
				$G_{16}$ & $10$ & $5$ \\ \hline
				$G_{20}$ & $6$ & $3$ \\ \hline
				$G_{22}$ & $4$ & $2$ \\ \hline
			\end{tabular}
			\caption{Order of a generator of the center and of reflections of some complex reflection groups of rank $2$. }
			\label{table:order}
		\end{table}
	\end{proof}
	
	Recall that for $m, n, p \in \mathbb{Z}_{\geq 1}$ and $p|m$, the complex reflection group $G(m,p,n)$ is the group of monomial $n \times n$-matrices with entries which are $m$\textsuperscript{th} roots of unity and product of the nonzero entries equal to a $(m/p)$\textsuperscript{th} root of unity. 
	
	\begin{prop}
		Let $W$ be an irreducible complex reflection group of rank $2$ containing reflections of order $2$ and assume that $z\in Z(W)$ has order $2$ (which forces $z=-I_2$). Let $A=\mathrm{Ref}(W)$. Then $[1,z]_A$ is a lattice, and the attached interval Garside group is isomorphic to $\mathbb{Z} \times F_{|R_2|/2}$, where $R_2$ denotes the set of reflections of order $2$ of $W$. This applies in particular if $W$ is $G_{12}$, $G_{13}$, $G_{22}$ or $G(m,p,2)$ with even $m$, and $z= - I_2$.
	\end{prop}
	
	\begin{proof}
		Since $W$ is irreducible, every central element consists of a scalar matrix. We thus have that $z= - I_2$. It follows that, letting $t\in R_2$, we have that $tz = zt$ is also a reflection, and it is distinct from $t$. If $t$ has order greater than $2$, then $tz$ cannot be a reflection. We thus have $\ell_A(z)=2$ and maximal chains in $[1,z]_A$ come by pairs: whenever $1 <_A t <_A tt'=z$ is a maximal chain (with $t\in R_2$), then $tt'=t't$ and $1 <_A t' <_A t't=z$ is also a maximal chain. We thus recover exactly the same interval group presentation as the one obtained for $I_2(|R_2|)$ in Section~\ref{sec_dih}. Note that under our assumptions $|R_2|$ is necessarily even, since the map $R_2 \rightarrow R_2, t \mapsto zt$ is involutive and without fixed points. 
	\end{proof}

	The examples given above provide situations where either the lattice property fails, or the interval Garside structure was already obtained in the real case. We finish this section by providing an example where the lattice property holds, and the obtained Garside structure is different from those obtained in the real case. 
	
	The examples obtained above suggest to consider a group where the chosen central element is of order $2$, and the rank of the group is greater than $2$. 
	
	\subsection{The complex reflection group $G(4,1,3)$}
	Let $W$ be the complex reflection group $G(4,1,3)$, consisting of monomial $3$ by $3$ matrices whose entries are $4$\textsuperscript{th} roots of unity. Let $A$ denote the set of reflections of $W$. A reflection in $W$ has eigenvalues $(1,1,\zeta)$, where $\zeta\in\{-1, \pm i\}$ and thus has order $2$ or $4$.
	
	We consider the element $z= - I_3\in W$. Since each reflection fixes a hyperplane, and the intersection of two hyperplanes in $\mathbb{C}^3$ is nonzero, every product of at most two reflections has a nonzero fixed vector. As $\operatorname{Fix}(z)=\{0\}$, we obtain $\ell_A(z)\geq 3$. On the other hand, since $z$ belongs to the subgroup $G(2,1,3)$ (which is the Weyl group of type $B_3$), where it is the longest element, it admits a reduced factorization into three reflections of $W_{B_3}$. As every reflection of $W_{B_3}$ is also a reflection of $G(4,1,3)$, we obtain $\ell_A(z)\le3$. Therefore,  $\ell_A(z)=3$.
	
	\begin{rmq}
		In a finite Coxeter group $W$, the reflection length of an element $w$ is equal the codimension of its space of fixed points $V^w$ (see~\cite[Section 2]{Carter}). As observed by Shi~\cite[Remark 2.3 (1)]{Shi} this also holds if $W$ is the complex reflection group $G(m,1,n)$. We shall not use this result here but it may be helpful to have it in mind. See also~\cite{Lewis_Wang} for a characterization of the elements $w$ in a complex reflection group $G(m,p,n)$ for which the reflection length is equal to the codimension of the fixed space. 
	\end{rmq}

	Since $\ell_A(z)=3$, every reduced reflection factorization of $z$ in $G(2,1,3)$ is also reduced with respect to $A$. Therefore $[1,z]_T\subseteq [1,z]_A$, where $T$ denotes the set of reflections of $G(2,1,3)$. We shall prove that $[1,z]_A$ is a lattice.
	
	We first show that no reflection in $W$ of order $4$ can lie in the interval $[1,z]_A$. Let $t$ be such a reflection. Then its eigenvalues are $(1,1,\mathbf{i})$ or $(1,1,-\mathbf{i})$. We consider the element $t^{-1}z=-t^{-1}$. Its eigenvalues are $(-1,-1,\mathbf{i})$ or $(-1,-1,-\mathbf{i})$. Therefore, $t^{-1}z$ has no nonzero fixed vector and, by the same argument as the one given above for $z$, it has reflection length at least $3$. However, is $t\leq_A z$, then $\ell_A(t^{-1} z)=\ell_A(z)-\ell_A(t)=3-1=2$, which is a contradiction.
	Hence every reflection in $[1,z]_A$ has order $2$. Therefore, for the study of the interval $[1,z]_A$, we may replace $A$ by the set $R$ of $2$-reflections. There are $15$ such reflections, listed below:
	\[
	\begin{array}{ccccc}
		\begin{pmatrix}
			-1&0&0\\
			0&1&0\\
			0&0&1
		\end{pmatrix}
		&
		\begin{pmatrix}
			1&0&0\\
			0&-1&0\\
			0&0&1
		\end{pmatrix}
		&
		\begin{pmatrix}
			1&0&0\\
			0&1&0\\
			0&0&-1
		\end{pmatrix}
		&
		\begin{pmatrix}
			0&1&0\\
			1&0&0\\
			0&0&1
		\end{pmatrix}
		&
		\begin{pmatrix}
			0&-\mathbf{i}&0\\
			\mathbf{i}&0&0\\
			0&0&1
		\end{pmatrix}
		\\[1.5em]
		
		\begin{pmatrix}
			0&-1&0\\
			-1&0&0\\
			0&0&1
		\end{pmatrix}
		&
		\begin{pmatrix}
			0&\mathbf{i}&0\\
			-\mathbf{i}&0&0\\
			0&0&1
		\end{pmatrix}
		&
		\begin{pmatrix}
			0&0&1\\
			0&1&0\\
			1&0&0
		\end{pmatrix}
		&
		\begin{pmatrix}
			0&0&-\mathbf{i}\\
			0&1&0\\
			\mathbf{i}&0&0
		\end{pmatrix}
		&
		\begin{pmatrix}
			0&0&-1\\
			0&1&0\\
			-1&0&0
		\end{pmatrix}
		\\[1.5em]
		
		\begin{pmatrix}
			0&0&\mathbf{i}\\
			0&1&0\\
			-\mathbf{i}&0&0
		\end{pmatrix}
		&
		\begin{pmatrix}
			1&0&0\\
			0&0&1\\
			0&1&0
		\end{pmatrix}
		&
		\begin{pmatrix}
			1&0&0\\
			0&0&-\mathbf{i}\\
			0&\mathbf{i}&0
		\end{pmatrix}
		&
		\begin{pmatrix}
			1&0&0\\
			0&0&-1\\
			0&-1&0
		\end{pmatrix}
		&
		\begin{pmatrix}
			1&0&0\\
			0&0&\mathbf{i}\\
			0&-\mathbf{i}&0
		\end{pmatrix}
	\end{array}
	\]

	We label the reflections of $R$ consistently with the notation used in Section~\ref{sub_bn} for type $B_3$. Thus, the reflections belonging to $B_3$ are denoted by $[i]$ ($1\leq i\leq 3$), $t_{i,j}$ and $r_{i,j}$ ($1\leq i<j\leq 3$). Let $(e_1,e_2,e_3)$ denote the canonical basis of $\mathbb{C}^3$. The reflection $[i]$ sends $e_i$ to $-e_i$ and fixes the remaining basis vectors, while for $u\in{t,r}$, the reflection $u_{i,j}$ sends $e_i$ to $\pm e_j$, $e_j$ to $\pm e_i$, and fixes $e_k$ for $k\neq i,j$, where the sign is $+$ if $u=t$ and $-$ if $u=r$.
	
	We denote the remaining reflections of $R$ by $v_{i,j}$ and $w_{i,j}$ ($1\leq i<j\leq 3$), according to their action on the canonical basis. They are defined by $v_{i,j}(e_i)=\mathbf{i}e_j$, $v_{i,j}(e_j)=-\mathbf{i}e_i$, and $v_{i,j}(e_k)=e_k$ for $k\neq i,j$; the definition of $w_{i,j}$ is obtained by replacing $\mathbf{i}$ with $-\mathbf{i}$. By abuse of notation, we use the same symbol for a reflection and for the corresponding endomorphism (or its matrix in the canonical basis).

	Assume now that $t_1t_2t_3=z$ is an $R$-reduced expression of $z$. Since $t_3$ is an involution and $z$ is central and involutive, we have that $t_1t_2=zt_3$ is an involution. Hence, since $t_1$ and $t_2$ have order two, we get $t_1t_2=t_2t_1$. We thus have $z=t_1 t_2 t_3 = t_2 t_1 t_3$. Since $z$ is central, we also have $z=t_1t_2t_3=t_3t_1t_2=t_2t_3t_1$, and the same argument applies to each of the other pairs of factors. It follows that, as in the real case, every $R$-reduced expression of $z$ consists of three pairwise commuting reflections.
	
	A straightforward computation shows that two reflections $t_1,t_2\in R$ commute if and only if ${t_1,t_2}$ is one of the following pairs:
	\begin{itemize}
		\item ${[i],[j]}$, where $i\neq j$;
		\item ${[i],x_{j,k}}$, where ${i,j,k}={1,2,3}$, $j<k$, and $x\in\{t,r,v,w\}$;
		\item ${t_{i,j},r_{i,j}}$, where $i<j$;
		\item ${v_{i,j},w_{i,j}}$, where $i<j$.
	\end{itemize}
	Multiplying the commuting pairs listed above yields all candidates for the elements of reflection length $2$ in $[1,z]_R$. These are the elements $[i][j]$ ($3$ elements) and $[i]x_{j,k}$, where ${i,j,k}={1,2,3}$, $j<k$, and $x\in\{t,r,v,w\}$ ($12$ elements). Thus we obtain $15$ elements.
	This list is exactly the list of elements of reflection length two in $[1,z]_R$. Indeed, if $y\in[1,z]_R$ has reflection length $2$, then $\ell_R(y^{-1}z)=1$, so $y^{-1}z=t$ for some $t\in R$. Since $z$ is central and $t$ is involutive, we have $y=zt=tz$. Hence the elements of reflection length $2$ in $[1,z]_R$ are precisely the elements $tz$, with $t\in R$. Since $|R|=15$, the above list contains all elements of reflection length $2$ in the interval. Note that $|[1,z]_R|=1+15+15+1=32$.

	\begin{lemma}
		The poset $[1,z]_A = [1,z]_R$ is a lattice. 
	\end{lemma}
	\begin{proof}
		Since the interval $[1,z]_R$ has a bottom element of reflection lenght $0$ and a top element of reflection length $3$, the only possible obstruction to the lattice property is a bowtie between elements of reflection length $1$ and elements of reflection length $2$, that is, a quadruple $(a,b,x,y)$ with $a\neq b$, $x\neq y$, $\ell_R(a)=\ell_R(b)=1$, $\ell_R(x)=\ell_R(y)=2$, $a\leq_R x$, $a\leq_R y$, $b\leq_R x$, and $b\leq_R y$.
		
		Let $a,b\in R$. If $a$ and $b$ both belong to one of the sets ${[i],[j],t_{i,j},r_{i,j},v_{i,j},w_{i,j}}$, where $i<j$, then there is a unique element of reflection length $2$ above both of them, namely $[i][j]$.
		Otherwise, $a$ and $b$ do not belong to a common set of this form, or equivalently, they do not fix a common vector of the canonical basis. Then either $\{a,b\}=\{[i],x_{j,k}\}$, where $\{i,j,k\}=\{1,2,3\}$, $j<k$, and $x\in\{t,r,v,w\}$, or $\{a,b\}=\{x_{i,j},y_{k,\ell}\}$, where $\{i,j,k,\ell\}=\{1,2,3\}$, $i<j$, $k<\ell$, and $x,y\in\{t,r,v,w\}$. In the first case, there is again a unique element of reflection length $2$ above both $a$ and $b$, namely $[i]x_{j,k}$. In the second case, there is no element of reflection length $2$ above both of them, and their least upper bound is $z$.
		
		Thus no bowtie can occur and, therefore, the poset $[1,z]_A=[1,z]_R$ is a lattice.
	\end{proof}
	
	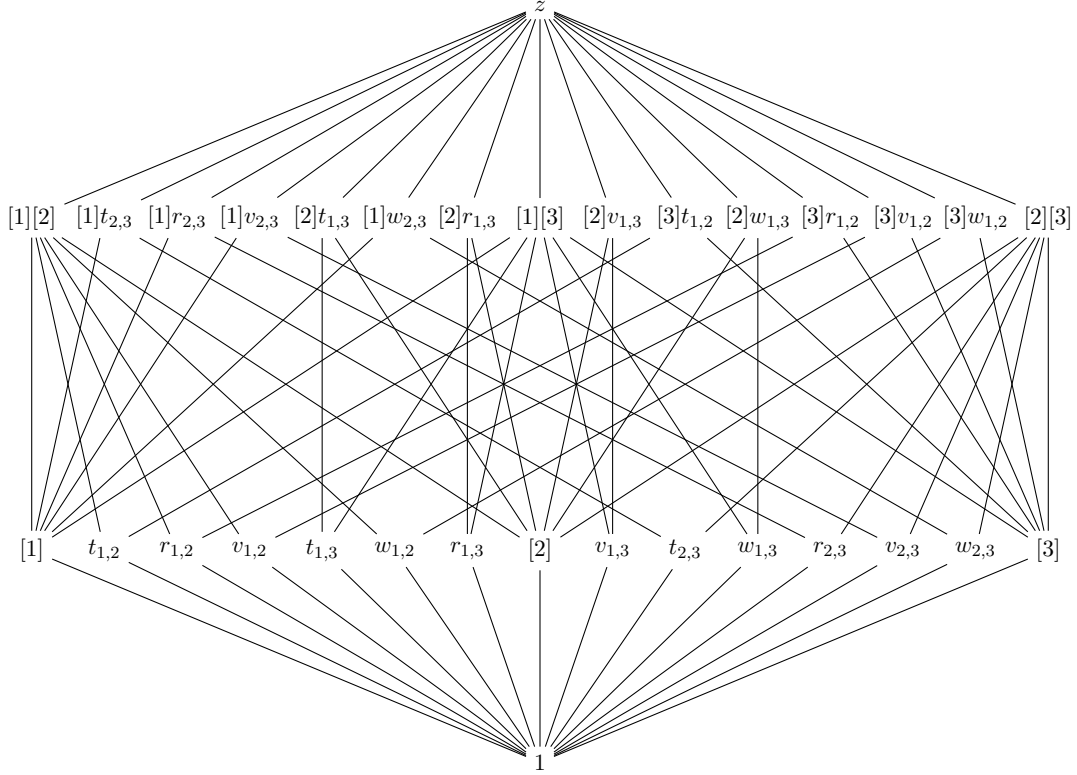
\begin{figure}
		\centering
		\begin{tikzpicture}[scale=0.8, every node/.style={scale=0.8}]
			
			\node (one) at (0,-1) {$1$};
			\node (z) at (0,11.5) {$z$};
			
			\node (b1) at (-8.4,2.5) {$[1]$};
			\node (b2) at (0,2.5) {$[2]$};
			\node (b3) at (8.4,2.5) {$[3]$};
			
			\node (t12) at (-7.2,2.5) {$t_{1,2}$};
			\node (r12) at (-6,2.5) {$r_{1,2}$};
			\node (v12) at (-4.8,2.5) {$v_{1,2}$};
			\node (w12) at (-2.4,2.5) {$w_{1,2}$};
			
			\node (t13) at (-3.6,2.5) {$t_{1,3}$};
			\node (r13) at (-1.2,2.5) {$r_{1,3}$};
			\node (v13) at (1.2,2.5) {$v_{1,3}$};
			\node (w13) at (3.6,2.5) {$w_{1,3}$};
			
			\node (t23) at (2.4,2.5) {$t_{2,3}$};
			\node (r23) at (4.8,2.5) {$r_{2,3}$};
			\node (v23) at (6,2.5) {$v_{2,3}$};
			\node (w23) at (7.2,2.5) {$w_{2,3}$};
			
			\node (b12) at (-8.4,8) {$[1][2]$};
			\node (b13) at (0,8) {$[1][3]$};
			\node (b23) at (8.4,8) {$[2][3]$};
			
			\node (b1t23) at (-7.2,8) {$[1]t_{2,3}$};
			\node (b1r23) at (-6,8) {$[1]r_{2,3}$};
			\node (b1v23) at (-4.8,8) {$[1]v_{2,3}$};
			\node (b1w23) at (-2.4,8) {$[1]w_{2,3}$};
			
			\node (b2t13) at (-3.6,8) {$[2]t_{1,3}$};
			\node (b2r13) at (-1.2,8) {$[2]r_{1,3}$};
			\node (b2v13) at (1.2,8) {$[2]v_{1,3}$};
			\node (b2w13) at (3.6,8) {$[2]w_{1,3}$};
			
			\node (b3t12) at (2.4,8) {$[3]t_{1,2}$};
			\node (b3r12) at (4.8,8) {$[3]r_{1,2}$};
			\node (b3v12) at (6,8) {$[3]v_{1,2}$};
			\node (b3w12) at (7.2,8) {$[3]w_{1,2}$};
			
			\foreach \x in {b1,b2,b3,t12,r12,v12,w12,t13,r13,v13,w13,t23,r23,v23,w23}
			\draw (one) -- (\x);
			
			\foreach \x in {b12,b13,b23,b1t23,b1r23,b1v23,b1w23,b2t13,b2r13,b2v13,b2w13,b3t12,b3r12,b3v12,b3w12}
			\draw (\x) -- (z);
			
			\foreach \x in {b1,b2,t12,r12,v12,w12}
			\draw (\x) -- (b12);
			
			\foreach \x in {b1,b3,t13,r13,v13,w13}
			\draw (\x) -- (b13);
			
			\foreach \x in {b2,b3,t23,r23,v23,w23}
			\draw (\x) -- (b23);
			
			\draw (b1) -- (b1t23);
			\draw (t23) -- (b1t23);
			\draw (b1) -- (b1r23);
			\draw (r23) -- (b1r23);
			\draw (b1) -- (b1v23);
			\draw (v23) -- (b1v23);
			\draw (b1) -- (b1w23);
			\draw (w23) -- (b1w23);
			
			\draw (b2) -- (b2t13);
			\draw (t13) -- (b2t13);
			\draw (b2) -- (b2r13);
			\draw (r13) -- (b2r13);
			\draw (b2) -- (b2v13);
			\draw (v13) -- (b2v13);
			\draw (b2) -- (b2w13);
			\draw (w13) -- (b2w13);
			
			\draw (b3) -- (b3t12);
			\draw (t12) -- (b3t12);
			\draw (b3) -- (b3r12);
			\draw (r12) -- (b3r12);
			\draw (b3) -- (b3v12);
			\draw (v12) -- (b3v12);
			\draw (b3) -- (b3w12);
			\draw (w12) -- (b3w12);
			
		\end{tikzpicture}
		\caption{The lattice $[1,z]_R$, where $z=-I_3$, in $G(4,1,3)$.}
	\end{figure}

	To identify the attached Garside group, we first list the relations arising from involutions of reflection length $2$. They are given by \begin{align*} [i] [j] &= [j] [i] = t_{i,j} r_{i,j} = r_{i,j} t_{i,j} = v_{i,j} w_{i,j}= w_{i,j} v_{i,j}, \ (i<j),\\
		[i] x_{j,k} &= x_{j,k} [i], \ \{i,j,k\}=\{1,2,3\}, j < k, x\in \{t,r,v,w\}.
	\end{align*}
	As for type $B_n$, we show that $M([1,z]_R)$ has a presentation with generators a copy of the set of reflections, and the above relations. It only remains to show that no new relations come from the element $z$, that is, that every two $R$-reduced expressions of $z$ can be related only by relations coming from involutions of reflection length $2$. This follows from the fact that, up to commutation of the letters, an $R$-reduced expression of $z$ is either given by $[1][2][3]$ or something of the form $[i] x_{j,k} y_{j,k}$ with $\{x,y\}= \{t,r\}$ or $\{x,y\} = \{ v,w\}$. But the latter can be related to the "canonical" reduced expression $[1][2][3]$ by using the relation $x_{j,k} y_{j,k} = [j][k]$ and commutation of adjacent factors, which all arise from involutions of reflection length $2$. 
	
	We thus have that $M([1,z]_R)$ has a presentation with generators $\underline{[i]}$, $\underline{x_{i,j}}$ ($i<j$, $x\in \{t,r,v,w\}$) and relations 
	\begin{align*}  \underline{[i]} \cdot \underline{[j]} = \underline{[j]} \cdot \underline{[i]} = \underline{t_{i,j}} \cdot \underline{r_{i,j}} = \underline{r_{i,j}} \cdot \underline{t_{i,j}} = \underline{v_{i,j}} \cdot \underline{w_{i,j}}=\underline{w_{i,j}} \cdot \underline{v_{i,j}}, & \quad \text{if } i < j, \\ \underline{[i]} \cdot \underline{t_{j,k}} = \underline{t_{j,k}} \cdot \underline{[i]}, \underline{[i]} \cdot \underline{r_{j,k}} = \underline{r_{j,k}} \cdot \underline{[i]}, & \\ \underline{[i]} \cdot \underline{v_{j,k}} = \underline{v_{j,k}} \cdot \underline{[i]}, \underline{[i]} \cdot \underline{w_{j,k}} = \underline{w_{j,k}} \cdot \underline{[i]} & \quad \text{if } |\{i,j,k\}| = 3, \ j < k. \end{align*}
	
	As for type $B_n$, in the corresponding Garside group $G([1,z]_R)$ we can obtain a presentation with a smaller number of generators by adding a central generator and removing some of the generators from $M([1,z]_R)$. Set $\Delta := \underline{[1]} \cdot \underline{[2]} \cdot \underline{[3]}$. We can then remove $[3]$ and $\underline{r_{i,j}}$ (as in type $B_n$) and $\underline{w_{i,j}}$, since $[3] = [2]^{-1} \cdot [1]^{-1} \cdot \Delta$, $\underline{r_{i,j}}= \underline{[i]} \cdot \underline{[j]} \cdot \underline{t_{i,j}}^{-1}$, $\underline{w_{i,j}}= \underline{[i]} \cdot \underline{[j]} \cdot \underline{v_{i,j}}^{-1}$. This yields a presentation which simplifies to
	$$ \mathbb{Z} \times \left\langle\, \begin{array}{l} [1], [2], \\ t_{i,j}, \quad 1 \leq i < j \leq 3, \\ v_{i,j}, \quad 1 \leq i < j \leq 3, \end{array} \ \middle| \ \begin{array}{ll} \underline{[1]} \cdot \underline{[2]} = \underline{[2]} \cdot \underline{[1]},   & \\[4pt] \underline{[1]} \cdot \underline{[2]} \cdot\underline{x_{1,2}} =\underline{x_{1,2}} \cdot \underline{[1]} \cdot \underline{[2]} & \text{if } x\in \{t, v\}, \\[4pt] \underline{[i]} \cdot \underline{x_{j,k}} = \underline{x_{j,k}} \cdot \underline{[i]}  & \text{if } |\{i,j,k\}| = 3, \ j < k, x\in \{t, v\} \end{array} \,\right\rangle.$$
	To compare with the interval group $G'$ obtained in type $B_3$ (see Example~\ref{ex_bn}), we rewrite this presentation. Setting $a=\underline{[1]}$, $b=\underline{[2]}$, $d=\underline{t_{1,2}}$, $c=\underline{t_{1,3}}$, $e=\underline{t_{2,3}}$, $d'=\underline{v_{1,2}}$, $c'=\underline{v_{1,3}}$, $e'=\underline{v_{2,3}}$, the presentation above becomes 	\[\mathbb{Z} \times \Biggl\langle 
	\begin{array}{l|clll}
		& ab=ba, \\
		a,b,c,c', d,d', e, e' & cb=bc, & c'b=bc', & abd=dab, & abd'=d'ab, \\
		& ae=ea,  & ae'=e'a                                            
	\end{array}
	\Biggr\rangle.\]
	Denoting this group $G$ we see that $G'$ embeds into $G$: the natural map $G' \longrightarrow G$ has a section obtained by setting $d'=c'=e'=1$.

\end{document}